\documentclass[11pt,letterpaper,reqno]{amsart}
\usepackage{amssymb,amsmath,amsthm,amsfonts,bbm,mathtools}

\usepackage{hyperref}
\hypersetup{pdfstartview={FitH}}
\usepackage{color}

\newtheorem{theorem}{Theorem}[section]
\newtheorem{lemma}[theorem]{Lemma}

\numberwithin{equation}{section}

\begin{document}

\title{A sharp nonlinear Hausdorff--Young inequality for small potentials}

\author[V. Kova\v{c}]{Vjekoslav Kova\v{c}}
\address{Vjekoslav Kova\v{c}, Department of Mathematics, Faculty of Science, University of Zagreb, Bijeni\v{c}ka cesta 30, 10000 Zagreb, Croatia}
\email{vjekovac@math.hr}

\author[D. Oliveira e Silva]{Diogo Oliveira e Silva}
\address{Diogo Oliveira e Silva, School of Mathematics, University of Birmingham, Edgbaston, Birmingham, B15 2TT England, and Hausdorff Center for Mathematics, Endenicher Allee 62, 53115 Bonn, Germany}
\email{d.oliveiraesilva@bham.ac.uk}

\author[J. Rup\v{c}i\'{c}]{Jelena Rup\v{c}i\'{c}}
\address{Jelena Rup\v{c}i\'{c}, Faculty of Transport and Traffic Sciences, University of Zagreb, Vukeli\'{c}eva 4, 10000 Zagreb, Croatia}
\email{jrupcic@fpz.hr}

\date{August 1, 2017. \emph{Revised}: April 16, 2018.}


\begin{abstract}
The nonlinear Hausdorff--Young inequality follows from the work of Christ and Kiselev.
Later Muscalu, Tao, and Thiele asked if the constants can be chosen independently of the exponent.
We show that the nonlinear Hausdorff--Young quotient admits an even better upper bound than the linear one, provided that the function is sufficiently small in the $\textup{L}^1$-norm.
The proof combines perturbative techniques with the sharpened version of the linear Hausdorff--Young inequality due to Christ.
\end{abstract}

\maketitle

\section{Introduction}

In this paper, we investigate the Hausdorff--Young inequality for a nonlinear version of the Fourier transform, and establish Theorem \ref{thm:main} below.
Before stating it precisely, we briefly discuss the linear case. Given a complex-valued integrable function $f$ on the real line, we normalize its Fourier transform as follows:
\[ \widehat{f}(\xi)=\int_{\mathbb{R}} f(x) e^{-2\pi i x\xi} \,{\rm d} x. \]
In this way, the Fourier transform is a contraction from $\textup{L}^1(\mathbb{R})$ to $\textup{L}^\infty(\mathbb{R})$ and it extends to a unitary operator on $\textup{L}^2(\mathbb{R})$.
Standard interpolation tools can then be used to show that, for any $p\in[1,2]$, the Fourier transform is also a contraction from $\textup{L}^p(\mathbb{R})$ to $\textup{L}^q(\mathbb{R})$, where $q=\frac p{p-1}$ denotes the exponent conjugate to $p$.
This is the content of the classical Hausdorff--Young inequality.
Its sharp version  was first established by Babenko \cite{B61} in the case when the exponent $q$ is an even integer, and then by Beckner \cite{B75} for general exponents. It states that, if $p\in[1,2]$, then
\begin{equation}\label{sharpHY}
\|\widehat{f}\|_{\textup{L}^q(\mathbb{R})}\leq {\bf B}_p \|f\|_{\textup{L}^p(\mathbb{R})}
\end{equation}
for every $f\in \textup{L}^p(\mathbb{R})$, where the optimal constant is given by
\begin{equation}\label{Beckner}
{\bf B}_p = p^{\frac1{2p}}q^{-\frac1{2q}}.
\end{equation}
An easy computation shows that Gaussians, i.e. functions of the form
\begin{equation}\label{Gaussian}
G(x)=c e^{-\alpha x^2+ v x}
\end{equation}
with $\alpha>0$ and $c,v\in\mathbb{C}$, turn inequality \eqref{sharpHY} into an equality.
In other words, Gaussians are \emph{extremizers} for inequality \eqref{sharpHY}.
In the converse direction, Lieb \cite{L90} has shown that all extremizers for inequality \eqref{sharpHY} are in fact Gaussians.
Recently, Christ \cite{C14} further refined inequality \eqref{sharpHY} by establishing the following sharpened version:
Given $p\in (1,2)$, there exists a constant $c_p>0$ such that, for every nonzero function $f\in \textup{L}^p(\mathbb{R})$,
\begin{equation}\label{Christ}
\| \widehat{f} \|_{\textup{L}^q(\mathbb{R})} \leq \Big( {\bf B}_p - c_p\frac{\textup{dist}^2_p(f,\mathfrak{G})}{\|f\|^2_{\textup{L}^p(\mathbb{R})}} \Big) \|f\|_{\textup{L}^p(\mathbb{R})}.
\end{equation}
Here, the distance from $f\in \textup{L}^p(\mathbb{R})$ to the set of all Gaussians, denoted $\mathfrak{G}$, is naturally \mbox{defined as}
\[ \textup{dist}_p(f,\mathfrak{G})=\inf_{G\in\mathfrak{G}} \|f-G\|_{\textup{L}^p(\mathbb{R})}. \]

We now describe the nonlinear setting of the present paper.
We are interested in the simplest nonlinear model of the Fourier transform, also known as the \emph{Dirac scattering transform} or the \emph{$\textup{SU(1,1)}$-scattering transform}.
To describe it precisely, take a compactly supported integrable function $f\colon\mathbb{R}\to\mathbb{C}$ and consider the initial-value problem
\begin{equation}\label{ODE}
\frac{\partial}{\partial x}
\begin{bmatrix} a(x,\xi) \\ b(x,\xi) \end{bmatrix}
 = \begin{bmatrix} 0 & \overline{f(x)} e^{2\pi i x \xi} \\ f(x) e^{-2\pi i x \xi} & 0 \end{bmatrix}
\begin{bmatrix} a(x,\xi) \\ b(x,\xi) \end{bmatrix},
\qquad
\begin{bmatrix} a(-\infty,\xi) \\ b(-\infty,\xi) \end{bmatrix}
 = \begin{bmatrix} 1 \\ 0 \end{bmatrix}.
\end{equation}
For each $\xi\in\mathbb{R}$, this problem has a unique solution $a(\cdot,\xi)$, $b(\cdot,\xi)$ in the class of absolutely continuous functions.
We simply write $a(\xi)$, $b(\xi)$ in place of the limits $a(+\infty,\xi)$, $b(+\infty,\xi)$, and define the \emph{nonlinear Fourier transform} of $f$ to be the function
\[ \mathbb{R}\to\mathbb{C}^2,\quad \xi \mapsto \begin{bmatrix} a(\xi) \\ b(\xi) \end{bmatrix}. \]
If $f$ vanishes outside some interval $I=[\alpha,\beta]$, then the initial condition in \eqref{ODE} translates into $a(\alpha,\xi)=1$, $b(\alpha,\xi)=0$, while $a(+\infty,\xi)$, $b(+\infty,\xi)$ can be interpreted as $a(\beta,\xi)$, $b(\beta,\xi),$ respectively.
The differential equation forces
\[ |a(\xi)|^2 - |b(\xi)|^2 = |a(\alpha,\xi)|^2 - |b(\alpha,\xi)|^2 = 1, \]
which in particular means that the size of the vector $[a(\xi) \; b(\xi)]^T$ is retained by the quantity $|a(\xi)|$ alone.
Occasionally it is more convenient to add an extra column and turn the above vector into a $2\times 2$ matrix belonging to the classical Lie group $\textup{SU}(1,1)$, see e.g. \cite{MS13} or \cite{MTT03a}.

Sources of motivation for considering this precise instance of the nonlinear Fourier transform include the eigenvalue problem for the Dirac operator, the study of completely integrable systems and scattering theory, and the Riemann--Hilbert problem; see the expository note \cite{T02} or the book \cite{MS13} for further information, and the lecture notes \cite{TT03} for several related examples in the discrete setting.
In those applications the input function $f$ is often referred to as a \emph{potential}.
The Dirac scattering transform is the simplest case of a more general transform, the AKNS--ZS nonlinear Fourier transform; see \cite{AKNS74}, \cite{ZS71} or \cite{MS13} for details.

There is a strong parallel between the nonlinear and the linear Fourier transforms.
It is a straightforward exercise to verify the following analogues of the symmetry rules for the linear Fourier transform, see \cite{T02}.
\begin{itemize}

\item \emph{Unimodular homogeneity}: If $f(x)=e^{i\theta}f_1(x)$, where $\theta\in\mathbb{R}$, then
\[ a(\xi)=a_1(\xi), \;\;\; b(\xi)=e^{i\theta}b_1(\xi). \]

\item \emph{Modulation symmetry}: If $f(x)=e^{2\pi i x \xi_0}f_1(x)$, where $\xi_0\in\mathbb{R}$, then
\[ a(\xi)=a_1(\xi-\xi_0), \;\;\; b(\xi)=b_1(\xi-\xi_0). \]

\item \emph{Translation symmetry}: If $f(x)=f_1(x-x_0)$, where $x_0\in\mathbb{R}$, then
\[ a(\xi)=a_1(\xi), \;\;\; b(\xi)=e^{-2\pi i x_0 \xi}b_1(\xi). \]

\item \emph{$\textup{L}^1$-normalized dilation symmetry}: If $f(x)=\lambda^{-1} f_1(\lambda^{-1}x)$, where $\lambda>0$, then
\[ a(\xi)=a_1(\lambda\xi), \;\;\; b(\xi)=b_1(\lambda\xi). \]

\item \emph{Conjugation symmetry}: If $f(x)=\overline{f_1(x)}$, then
\[ a(\xi)=\overline{a_1(-\xi)}, \;\;\; b(\xi)=\overline{b_1(-\xi)}. \]

\item \emph{A substitute for additivity}: If $f(x)=f_1(x)+f_2(x)$, where the support of $f_1$ lies to the left of that of $f_2$, then
\[ a(\xi) = a_1(\xi)a_2(\xi) + b_1(\xi)\overline{b_2(\xi)}, \;\;\; b(\xi) = a_1(\xi)b_2(\xi) + b_1(\xi)\overline{a_2(\xi)}. \]

\end{itemize}

So far we have only defined the nonlinear Fourier transform for compactly supported $\textup{L}^1$ functions, but this can be easily extended to include general integrable functions $f:\mathbb{R}\to\mathbb{C}$ as follows.
Given $f\in \textup{L}^1(\mathbb{R})$ and $R>0$, denote by $\xi\mapsto [a_R(\xi)\; b_R(\xi)]^T$ the nonlinear Fourier transform of the truncated function
\[ f_R(x) := \begin{cases} f(x), & \text{if } -R\leq x\leq R, \\ 0, & \text{otherwise}. \end{cases} \]
Then the nonlinear Fourier transform of $f$ is defined via
\begin{equation}\label{noncomplimits}
a(\xi):= \lim_{R\to\infty} a_R(\xi), \;\;\; b(\xi):=\lim_{R\to\infty} b_R(\xi).
\end{equation}
Note that these limits exist, for every $\xi\in\mathbb{R}$.
Moreover, this definition coincides with the one from \cite{MS13} or \cite{T02}, formulated in terms of uniformly convergent multilinear expansions.
The assignment $f\mapsto(a,b)$ retains the six symmetry rules mentioned above.

We proceed to describe some nonlinear analogues of standard estimates for the linear Fourier transform.
The nonlinear Riemann--Lebesgue estimate, which follows easily from Gr\"onwall's inequality, states that
\[ \big\| (\log |a(\xi)|^2)^\frac{1}{2} \big\|_{\textup{L}^{\infty}_\xi(\mathbb{R})}
\leq \big\| \log(|a(\xi)|\!+\!|b(\xi)|) \big\|_{\textup{L}^{\infty}_\xi(\mathbb{R})}
\leq \|f\|_{\textup{L}^1(\mathbb{R})}, \]
for every potential $f\in\textup{L}^1(\mathbb{R})$.
The nonlinear Plancherel identity, which is a well-known scattering identity that can be established via complex contour integration (see e.g.\@ the Appendix in \cite{MTT03a}), states that
\[ \big\| (\log |a(\xi)|^2)^\frac{1}{2} \big\|_{\textup{L}^2_\xi(\mathbb{R})} = \|f\|_{\textup{L}^2(\mathbb{R})},\]
for every $f\in\textup{L}^1(\mathbb{R})\cap\textup{L}^2(\mathbb{R})$.
The latter equality can even be used to extend the definition of $a,b$ for functions $f$ that are only square-integrable.
In this case the existence of the limits \eqref{noncomplimits} for a.e.\@ $\xi\in\mathbb{R}$ is a well-known open problem (see \cite{CK01b}, \cite{MTT03a}, \cite{MTT03b}), so it cannot provide a straightforward extension of the nonlinear Fourier transform as before.
Even if interpolation is not available in the present nonlinear setting, the work of Christ and Kiselev on the spectral theory of one-dimensional Schr\"odinger operators \cite{CK01a}, \cite{CK01b} establishes a version of the nonlinear Hausdorff--Young inequality  which translates into the present context as follows: If $p\in[1,2]$, then there exists a constant $C_p<\infty$, such that
\begin{equation}\label{nonlinHY}
\big\| (\log |a(\xi)|^2)^\frac{1}{2} \big\|_{\textup{L}^q_\xi(\mathbb{R})} \leq C_p \|f\|_{\textup{L}^p(\mathbb{R})},
\end{equation}
for every potential $f\in\textup{L}^1(\mathbb{R})\cap\textup{L}^p(\mathbb{R})$.
An interesting question raised in \cite{MTT03a} is whether the constants $C_p$ can be chosen uniformly in $p$, as $p\to 2^{-}$.
This has been confirmed in a particular toy model in \cite{K12}, but remains an open problem in its full generality.
By considering Gaussian potentials $G(x) = c \exp(-\pi x^2)$
and linearizing as $c\to 0$, one can check that the constant in \eqref{nonlinHY} is at least as large as Beckner's constant \eqref{Beckner}, i.e.\@ $C_p\geq {\bf B}_p$. One may be tempted to conjecture that the optimal constant in \eqref{nonlinHY} is actually $C_p={\bf B}_p$.

While these questions are left open by the present work, we are able to provide some further supporting evidence of their validity by considering the behaviour of the nonlinear Hausdorff--Young ratio for sufficiently small potentials.
The main result of this paper is the following theorem.

\begin{theorem}\label{thm:main}
 Let $p\in (1,2)$, $q=\frac p{p-1}$, and  ${\bf B}_p = p^{\frac1{2p}}q^{-\frac1{2q}}$. Let $A>0$ and $0<\lambda<1$.
 Then there exist  $\delta>0$ and $\varepsilon>0$, depending on $p,A,\lambda$, with the following property:
If $S\subset\mathbb{R}$ is a measurable subset  of Lebesgue measure $0<|S|<\infty$, and $f\colon\mathbb{R}\to\mathbb{C}$ is a measurable function satisfying
\begin{itemize}
\item[(i)] $\|f\|_{\textup{L}^1(\mathbb{R})}\leq\delta$,
\item[(ii)]  $\|f\|_{\textup{L}^1(S)}\geq\lambda\|f\|_{\textup{L}^1(\mathbb{R})}$,
\item[(iii)] $\|f\|_{\textup{L}^p(\mathbb{R})}^p\leq A |S|^{1-p} \|f\|_{\textup{L}^1(\mathbb{R})}$,
\end{itemize}
then
\begin{equation}\label{mainineq}
\big\| (\log |a(\xi)|^2)^\frac{1}{2} \big\|_{\textup{L}^q_\xi(\mathbb{R})}
\leq
\big( {\bf B}_p - \varepsilon\|f\|_{\textup{L}^1(\mathbb{R})}^2 \big) \|f\|_{\textup{L}^p(\mathbb{R})}.
\end{equation}
\end{theorem}

\noindent A few remarks may help to further orient the reader.
\begin{itemize}

\item
Inequality \eqref{mainineq} implies \eqref{nonlinHY} with an optimal constant $C_p={\bf B}_p$, but only within the restricted collection of functions considered in Theorem~\ref{thm:main}, which in particular depends on $p$.
That way the theorem does not claim uniform boundedness of the constants in \eqref{nonlinHY} for any particular family of functions.
It rather fixes the value of $p$ and shows that the nonlinear Hausdorff--Young inequality beats the linear one in the asymptotic regime when $\|f\|_{\textup{L}^1(\mathbb{R})}\to 0$.

\item
For each function $f\in \textup{L}^1(\mathbb{R})\cap\textup{L}^p(\mathbb{R})$ there exist parameters $A,\lambda$ and a subset $S\subset\mathbb{R}$ of finite positive measure for which the above conditions (ii) and (iii) are satisfied.
Moreover, by replacing $f$ with $cf$ for any $c\in\mathbb{C}$ with sufficiently small modulus, we obtain functions that additionally satisfy condition (i).
Then from inequality \eqref{mainineq} we see that the nonlinear Hausdorff--Young ratio for these functions falls strictly below ${\bf B}_p$.
However, we preferred to formulate a stronger result than the one dealing with sufficiently small multiples of a fixed potential. Theorem~\ref{thm:main} shows a concrete improvement for all functions in the collection determined by the exponent $p$ and the fixed parameters $A,\lambda$.

\item
Fix a constant $D>0$ and a set $S\subset\mathbb{R}$ of finite positive measure.
A collection of measurable functions $f:\mathbb{R}\to\mathbb{C}$ supported on $S$ and bounded in absolute value by $D|S|^{-1}$ clearly satisfies conditions (ii) and (iii) above.
Theorem~\ref{thm:main} applies to those functions $f$ from the latter collection that also satisfy the $\textup{L}^1$-bound (i).
This particular case of the theorem is already nontrivial.
Indeed, we begin the proof in \S \ref{sec:prelim} by reducing to the case when $f$ is compactly supported.

\item
Between the lines of the first part of the proof below (see \S \ref{sec:far}), one can easily obtain
\begin{equation}\label{altineq}
\big\| (\log |a(\xi)|^2)^\frac{1}{2} \big\|_{\textup{L}^q_\xi(\mathbb{R})}
\leq
{\bf B}_p \exp(\|f\|_{\textup{L}^1(\mathbb{R})}) \|f\|_{\textup{L}^p(\mathbb{R})}
\end{equation}
for each $f\in\textup{L}^1(\mathbb{R})$, i.e.\@ without any restrictions on the potential $f$ other than integrability.
Estimate \eqref{altineq} provides a cheap version of \eqref{mainineq}, with ${\bf B}_p - \varepsilon\|f\|_{\textup{L}^1(\mathbb{R})}^2$ replaced by ${\bf B}_p + O(\|f\|_{\textup{L}^1(\mathbb{R})})$, as $\|f\|_{\textup{L}^1(\mathbb{R})}\to0$, but the interest lies, of course, in obtaining the estimate with a negative sign.
However, \eqref{altineq} at least shows that the mere uniformity of the constants $C_p$ is trivial for potentials that are controlled in the $\textup{L}^1$-norm.

\item
Both the list of assumptions (i)--(iii) and inequality \eqref{mainineq} are invariant under $\textup{L}^1$-normalized dilations applied to $f$.
Consequently, the $\textup{L}^1$-norm cannot be replaced by the $\textup{L}^p$-norm on the right-hand side of \eqref{mainineq}.
If that were indeed possible, then one could consider $\textup{L}^1$-preserving scalings so that $\|f\|_{\textup{L}^p(\mathbb{R})}\to+\infty$ and the right-hand side would eventually become negative -- a clear contradiction.

\end{itemize}

Let us briefly comment on the proof of Theorem \ref{thm:main}, which spans over the next four sections.
We start by reducing the analysis to that of compactly supported functions, for which the nonlinear Fourier transform is defined more directly.
The upshot is then that for most choices of $f$ one can simply estimate the error arising from linearization, while for the remaining ones we verify that the nonlinear effect actually improves the estimate.
To implement this strategy, we set up a case distinction, depending on whether the function $f$ is far or close to the set of Gaussians, in an appropriate sense.
The former case is the subject of \S \ref{sec:far}, where we invoke Christ's sharpened Hausdorff--Young inequality \eqref{Christ} in order to absorb the error terms coming from linearization.
The latter case is the subject of \S \ref{sec:close}, where we use a perturbative argument to expand the functional in question around a suitable Gaussian that provides a good upper bound for $f$ in \emph{both} the $\textup{L}^p$- and the $\textup{L}^1$-norms.
We are naturally led to study a certain quartic operator $\mathcal{Q}$ and its quadrilinear variant $\Phi$, which enjoy various symmetries. The operator $\mathcal{Q}$ can in turn be pointwise dominated by a power of the maximally truncated Fourier transform, defined as follows:
\begin{equation}\label{MaxTruncFT}
(\mathcal{F}_\ast f)(\xi) = \sup_{I} \Big|\int_I f(x) e^{-2\pi i x\xi} \,{\rm d} x\Big|,
\end{equation}
where the supremum is taken over all intervals $I\subseteq\mathbb{R}$.
The classical Menshov--Paley--Zygmund inequality (see \cite{MS13}) states that, for every $p\in (1,2)$, there exists $M_p<\infty$ such that, for every $f\in \textup{L}^p(\mathbb{R})$,
\begin{equation}\label{MPZ}
\|\mathcal{F}_\ast f\|_{\textup{L}^q(\mathbb{R})} \leq M_p \|f\|_{\textup{L}^p(\mathbb{R})}.
\end{equation}
The argument makes crucial use of estimate \eqref{MPZ} in order to control the error terms, and to verify that the second order variation about the aforementioned Gaussian has the correct sign.
Proofs of several technical lemmata are deferred to \S \ref{sec:proofs}.

Numerical evidence seems to indicate that the inequality
\[ \big\| (\log |a(\xi)|^2)^\frac{1}{2} \big\|_{\textup{L}^q_\xi(\mathbb{R})} \leq \big\| \widehat{f} \big\|_{\textup{L}^q(\mathbb{R})} \]
does not hold in general, even for bounded, compactly supported potentials $f$.
If this is indeed the case, then presumably Theorem~\ref{thm:main} cannot be established just by regarding \eqref{mainineq} as a small perturbation of \eqref{sharpHY}.

\smallskip

\noindent {\bf Notation.}
When $x,y$ are real numbers, we write $x=O(y)$ or $x\lesssim y$ if there exists a finite absolute constant $C$ such that $|x|\leq C|y|$. If we want to make explicit the dependence of the constant $C$  on some parameter $\alpha$, we write $x=O_\alpha(y)$ or $x\lesssim_\alpha y$.
We also write $x\vee y=\max\{x,y\}$ and $x\wedge y=\min\{x,y\}$.
The real and imaginary parts of a complex number $z$ are denoted by $\textup{Re}(z)$ and $\textup{Im}(z)$.
The indicator function of a set $E\subseteq\mathbb{R}$ is denoted by $\mathbbm{1}_E$.
Throughout the paper it will be understood that all constants may depend on the admissible parameters $p, A, \lambda$.

\section{First reduction}\label{sec:prelim}

In the proof of Theorem \ref{thm:main}, there is no loss of generality in restricting attention to compactly supported functions.
Indeed, once this special case is settled, one may invoke inequality \eqref{mainineq} with $f$ replaced by its truncation $f\mathbbm{1}_{[-R,R]}$, and apply Fatou's lemma to the left-hand side of \eqref{mainineq} as $R\to+\infty$.
To do so, one should further observe that there exists $R_0>0$ such that the functions $\{f\mathbbm{1}_{[-R,R]}\}_{R\geq R_0}$ satisfy conditions (ii) and (iii) of Theorem \ref{thm:main} with new parameters $A^*:=2A$,  $\lambda^*:=\tfrac{\lambda}2$, and the same measurable subset $S\subset\mathbb{R}$.
Indeed, if $R_0>0$ is sufficiently large, then
\[ \|f\mathbbm{1}_{[-R,R]}\|_{\textup{L}^1(\mathbb{R})} \geq \frac{1}{2} \|f\|_{\textup{L}^1(\mathbb{R})},
\text{ and }
\|f\mathbbm{1}_{[-R,R]}\|_{\textup{L}^1(S)} \geq \frac{1}{2} \|f\|_{\textup{L}^1(S)} \geq \frac{\lambda}{2} \|f\mathbbm{1}_{[-R,R]}\|_{\textup{L}^1(\mathbb{R})}, \]
for every $R\geq R_0$.
As a consequence,
\[ \|f\mathbbm{1}_{[-R,R]}\|_{\textup{L}^p(\mathbb{R})}^p \leq \|f\|_{\textup{L}^p(\mathbb{R})}^p \leq A |S|^{1-p} \|f\|_{\textup{L}^1(\mathbb{R})} \leq 2A |S|^{1-p} \|f\mathbbm{1}_{[-R,R]}\|_{\textup{L}^1(\mathbb{R})}, \]
as claimed. Compact support will be a standing assumption for the rest of the argument.

\section{Far from Gaussians}\label{sec:far}

Let $f\colon\mathbb{R}\to\mathbb{C}$ satisfy the assumptions of Theorem \ref{thm:main}.
Our first task is to make precise what it means for $f$ to be far away from the set of Gaussians.
The following notion will be suitable for our purposes:
Assume that
\[ \textup{dist}_p(f,\mathfrak{G}) \geq \gamma \|f\|_{\textup{L}^p(\mathbb{R})}, \]
where $\gamma$ is shorthand notation for
\begin{equation}\label{gamma}
\gamma = \Big(\frac{8 {\bf B}_p \|f\|_{\textup{L}^1(\mathbb{R})}}{c_p}\Big)^{\frac12}.
\end{equation}
Here, ${\bf B}_p$ denotes Beckner's constant \eqref{Beckner} and $c_p$ is the constant promised by Christ's refinement \eqref{Christ}.
This precise choice of $\gamma$ will become clear as the proof unfolds.
In particular, note that $\gamma\lesssim\delta^\frac{1}{2}$.
Going back to the defining ODE \eqref{ODE}, it is straightforward to check that the functions $a$ and $b$ satisfy the integral equations
\begin{align*}
& a(x,\xi) = 1 + \int_{-\infty}^x \overline{f(t)} e^{2\pi i t \xi} b(t,\xi) \,{\rm d} t, \\
& b(x,\xi) = \int_{-\infty}^x f(t) e^{-2\pi i t \xi} a(t,\xi) \,{\rm d} t.
\end{align*}
Adding the two equations, the triangle inequality yields
\[ |b(x,\xi)| + |a(x,\xi) - 1| \leq \Big|\int_{-\infty}^x f(t) e^{-2\pi i t \xi} \,{\rm d} t\Big| + \int_{-\infty}^x |f(t)| \big(|b(t,\xi)| + |a(t,\xi)-1|\big) \,{\rm d} t. \]
Taking the $\textup{L}^q$-norm in $\xi$ and invoking Minkowski's integral inequality, we obtain
\[ \big\| |b(x,\xi)| + |a(x,\xi) - 1| \big\|_{\textup{L}^q_\xi(\mathbb{R})}
\leq \big\|\widehat{f\mathbbm{1}}_{(-\infty,x]}\big\|_{\textup{L}^q(\mathbb{R})}
+ \int_{-\infty}^x |f(t)| \big\| |b(t,\xi)| + |a(t,\xi) - 1| \big\|_{\textup{L}^q_\xi(\mathbb{R})} \,{\rm d} t. \]
To estimate the quantity $\|\widehat{f\mathbbm{1}}_{(-\infty,x]}\|_{\textup{L}^q(\mathbb{R})}$, we further split the analysis into two cases.
\smallskip

\noindent
\emph{Case 1.} $\|f\mathbbm{1}_{(x,+\infty)}\|_{\textup{L}^p(\mathbb{R})} < \frac{\gamma}2 \|f\|_{\textup{L}^p(\mathbb{R})}.$
In this case, for each $G\in\mathfrak{G}$ we have that
\[ \|f\mathbbm{1}_{(-\infty,x]}-G\|_{\textup{L}^p(\mathbb{R})}
\geq \|f-G\|_{\textup{L}^p(\mathbb{R})} - \|f\mathbbm{1}_{(x,+\infty)}\|_{\textup{L}^p(\mathbb{R})}
\geq \frac{\gamma}2 \|f\|_{\textup{L}^p(\mathbb{R})}. \]
It follows that
\[ \textup{dist}_p(f\mathbbm{1}_{(-\infty,x]},\mathfrak{G}) \geq \frac{\gamma}2 \|f\mathbbm{1}_{(-\infty,x]}\|_{\textup{L}^p(\mathbb{R})}, \]
and Christ's improved Hausdorff--Young inequality yields
\[ \big\|\widehat{f\mathbbm{1}}_{(-\infty,x]}\big\|_{\textup{L}^q(\mathbb{R})}
\leq \Big( {\bf B}_p - c_p\big(\frac{\gamma}2\big)^2 \Big) \|f\mathbbm{1}_{(-\infty,x]}\|_{\textup{L}^p(\mathbb{R})}
\leq \widetilde{\bf B}_p \|f\|_{\textup{L}^p(\mathbb{R})},\]
where $\widetilde{\bf B}_p$ is defined as
$\widetilde{\bf B}_p = {\bf B}_p - c_p(\frac{\gamma}2)^2.$
Note that $\widetilde{\bf B}_p$ is not really a constant, since $\gamma$ depends on the $\textup{L}^1$-norm of $f$.
\smallskip

\noindent
\emph{Case 2.} $\|f\mathbbm{1}_{(x,+\infty)}\|_{\textup{L}^p(\mathbb{R})} \geq \frac{\gamma}2 \|f\|_{\textup{L}^p(\mathbb{R})}.$
In this case, the sharp Hausdorff--Young inequality \eqref{sharpHY} yields
\[ \big\|\widehat{f\mathbbm{1}}_{(-\infty,x]}\big\|_{\textup{L}^q(\mathbb{R})}
\leq {\bf B}_p \|f\mathbbm{1}_{(-\infty,x]}\|_{\textup{L}^p(\mathbb{R})}
\leq {\bf B}_p \Big(1-\big(\frac{\gamma}{2}\big)^p\Big)^{\frac{1}{p}} \|f\|_{\textup{L}^p(\mathbb{R})}. \]
Bernoulli's inequality can then be invoked to verify that
\begin{equation}\label{ineqBp}
{\bf B}_p \Big(1-\big(\frac{\gamma}{2}\big)^p\Big)^{\frac{1}{p}} \leq {\bf B}_p \Big(1-\frac{1}{p}\big(\frac{\gamma}{2}\big)^p\Big) \leq \widetilde{\bf B}_p,
\end{equation}
provided $\delta$ is chosen to be sufficiently small.
Here we also use that $p<2$.
\smallskip

In both cases, we obtain
\[ \big\| |b(x,\xi)| + |a(x,\xi) - 1| \big\|_{\textup{L}^q_\xi(\mathbb{R})}
\leq \widetilde{\bf B}_p \|f\|_{\textup{L}^p(\mathbb{R})} + \int_{-\infty}^x |f(t)| \big\| |b(t,\xi)| + |a(t,\xi) - 1| \big\|_{\textup{L}^q_\xi(\mathbb{R})} \,{\rm d} t. \]
Gr\"{o}nwall's lemma then implies
\[ \big\| |b(x,\xi)| + |a(x,\xi) - 1| \big\|_{\textup{L}^q_\xi(\mathbb{R})} \leq \widetilde{\bf B}_p
\exp\Big({\int_{-\infty}^{x}|f(t)|\,{\rm d} t}\Big) \|f\|_{\textup{L}^p(\mathbb{R})}. \]
Letting $x\to+\infty$ and estimating $(\log|a|^2)^\frac{1}{2}\leq|b|$, we finally have that
\[ \big\| (\log |a(\xi)|^2)^\frac{1}{2} \big\|_{\textup{L}^q_\xi(\mathbb{R})}
\leq \widetilde{\bf B}_p \exp(\|f\|_{\textup{L}^1(\mathbb{R})}) \|f\|_{\textup{L}^p(\mathbb{R})}. \]
The obtained inequality shows that, in this case, the only loss in passing from the linear to the nonlinear setting amounts to the exponential factor, which tends to $1$ as $\|f\|_{\textup{L}^1(\mathbb{R})}\to 0$.
Recall the choice of $\gamma$ from \eqref{gamma}.
It remains to choose $\delta, \varepsilon>0$ small enough, so that \eqref{ineqBp} holds, and
\begin{align*}
\widetilde{\bf B}_p \exp(\|f\|_{\textup{L}^1(\mathbb{R})})
& = {\bf B}_p (1  - 2 \|f\|_{\textup{L}^1(\mathbb{R})}) \big(1 + \|f\|_{\textup{L}^1(\mathbb{R})} + O(\|f\|_{\textup{L}^1(\mathbb{R})}^2)\big) \\
& = {\bf B}_p \big(1 - \|f\|_{\textup{L}^1(\mathbb{R})} + O(\|f\|_{\textup{L}^1(\mathbb{R})}^2)\big)
\leq {\bf B}_p - \varepsilon\|f\|_{\textup{L}^1(\mathbb{R})}^2.
\end{align*}
Thus we have verified the desired inequality \eqref{mainineq} in the case when the function $f$ is far from the Gaussians.
We analyse the complementary situation in the next section, where an additional smallness condition will be imposed on $\delta, \varepsilon$.

\section{Close to Gaussians}\label{sec:close}

We are now working under the assumption
\[ \textup{dist}_p(f,\mathfrak{G}) < \gamma \|f\|_{\textup{L}^p(\mathbb{R})}, \]
where $\gamma$ was defined in \eqref{gamma}.
In particular, $\gamma\lesssim\|f\|_{\textup{L}^1(\mathbb{R})}^{\frac{1}{2}}$.
Thus, there exists a Gaussian $G\in\mathfrak{G}$, such that
\begin{equation}\label{Gapproxf}
\|f-G\|_{\textup{L}^p(\mathbb{R})} < \gamma \|f\|_{\textup{L}^p(\mathbb{R})}.
\end{equation}
This readily implies
\begin{equation}\label{approxLp}
\|f\|_{\textup{L}^p(\mathbb{R})} \leq 2\|G\|_{\textup{L}^p(\mathbb{R})},
\end{equation}
provided $\gamma<\frac12$.
Recall our working assumptions (i), (ii), and (iii) from the statement of Theorem \ref{thm:main}.
Under these conditions, using H\"{o}lder's inequality we may estimate
\begin{align*}
\|f-G\|_{\textup{L}^1(S)}
\leq |S|^{1-\frac1p}\|f-G\|_{\textup{L}^p(S)}
\lesssim |S|^{1-\frac1p}\|f\|_{\textup{L}^1(\mathbb{R})}^{\frac{1}{2}} \|f\|_{\textup{L}^p(\mathbb{R})}
\lesssim_A \|f\|_{\textup{L}^1(\mathbb{R})}^{\frac{1}{2}+\frac{1}{p}}
\lesssim_\lambda \delta^{\frac{1}{p}-\frac{1}{2}}\|f\|_{\textup{L}^1(S)}.
\end{align*}
Since $\frac{1}{p}-\frac{1}{2}>0$, it follows that
\[ \|f\|_{\textup{L}^1(S)} \leq 2\|G\|_{\textup{L}^1(S)}, \]
provided $\delta$ is small enough.
This readily implies
\begin{equation}\label{approxL1}
\|f\|_{\textup{L}^1(\mathbb{R})} \leq  \tfrac{2}{\lambda} \|G\|_{\textup{L}^1(\mathbb{R})}.
 \end{equation}
Inequalities \eqref{approxLp} and \eqref{approxL1} ensure that the Gaussian $G$ is a good upper bound for the function $f$ on the whole real line, both in the $\textup{L}^p$ and in the $\textup{L}^1$ senses.

We now proceed to derive the first nontrivial term in the expansion of the left-hand side of inequality \eqref{mainineq} for any $q>2$.
The first observation is that
\begin{equation}\label{relar}
\log (|a(x,\xi)|^2) = -\log(1-|r(x,\xi)|^2),
\end{equation}
where the reflection coefficient $r = \frac{b}{a}$ satisfies Riccati's differential equation
\begin{align*}
\frac{\partial r}{\partial x} (x,\xi) = f(x) e^{-2\pi i x \xi} - \overline{f(x)} e^{2\pi i x \xi} r(x,\xi)^2, \qquad r(-\infty,\xi) = 0.
\end{align*}
In other words, the reflection coefficient is given by the integral equation
\begin{equation}\label{intRicatti}
r(x,\xi) = \int_{-\infty}^x f(t) e^{-2\pi i t \xi} \,{\rm d} t - \int_{-\infty}^x \overline{f(t)} e^{2\pi i t \xi} r(t,\xi)^2 \,{\rm d} t.
\end{equation}
From identity \eqref{relar} we obtain
\[ \frac{\partial}{\partial x} \log (|a(x,\xi)|^2)
=\frac{2\,\textup{Re}\big(\overline{r(x,\xi)}\,\frac{\partial r}{\partial x}(x,\xi)\big)}{1-|r(x,\xi)|^2}
=2\,\textup{Re}\Big(f(x)e^{-2\pi i x \xi} \,\overline{r(x,\xi)}\Big), \]
i.e.
\begin{equation}\label{loguseful}
\log (|a(\xi)|^2) = 2\,\textup{Re} \int_\mathbb{R} f(x_1) e^{-2\pi i x_1 \xi} \,\overline{r(x_1,\xi)} \,\,{\rm d} x_1 .
\end{equation}
Using the integral equation \eqref{intRicatti} to substitute for $r(x_1,\xi)$ on the right-hand side of \eqref{loguseful},
\begin{align*}
\log (|a(\xi)|^2) & = 2\,\textup{Re} \int_{\mathbb{R}} \Big(\int_{-\infty}^{x_1} f(x_1)\overline{f(x_2)} e^{-2\pi i (x_1-x_2)\xi} \,\,{\rm d} x_2 \Big) \,\,{\rm d} x_1 \\
& \quad - 2\,\textup{Re}\int_{\mathbb{R}}\Big(\int_{-\infty}^{x_1} f(x_1)f(x_2) e^{-2\pi i (x_1+x_2)\xi}\,\overline{r(x_2,\xi)^2}\,\,{\rm d} x_2\Big)\,\,{\rm d} x_1.
\end{align*}
The first summand on the right-hand side can be recognized as
\[ \int\limits_{\{x_1>x_2\}} f(x_1)\overline{f(x_2)} e^{-2\pi i (x_1-x_2)\xi} \,\,{\rm d} x_1 \,{\rm d} x_2
+ \int\limits_{\{x_2>x_1\}} \overline{f(x_2)}f(x_1) e^{2\pi i (x_2-x_1)\xi} \,\,{\rm d} x_1 \,{\rm d} x_2
= |\widehat{f}(\xi)|^2. \]
Repeating this procedure once again, i.e.\@ substituting for $r(x_2,\xi)$ and symmetrizing in the variables $x_1$ and $x_2$, we conclude that
\begin{equation}\label{logexpansion}
\log (|a(\xi)|^2) = |\widehat{f}(\xi)|^2 - (\mathcal{Q} f)(\xi) + (\mathcal{E} f)(\xi),
\end{equation}
where $\mathcal{Q}$ is the quartic operator defined as
\[ (\mathcal{Q} f)(\xi)  = \textup{Re} \int\limits_{\{(x_1\wedge x_2)>(x_3\vee x_4)\}}\!\!
f(x_1)f(x_2) \overline{f(x_3)f(x_4)} \, e^{2\pi i (-x_1-x_2+x_3+x_4) \xi} \, \,{\rm d} x_1 \,{\rm d} x_2 \,{\rm d} x_3 \,{\rm d} x_4 \]
and $\mathcal{E}$ is the nonlinear operator given by
\begin{align*}
& (\mathcal{E} f)(\xi)
= 2\,\textup{Re} \!\!\!\!\!\int\limits_{\{(x_1\wedge x_2)>(x_3\vee x_4)\}}\!\!\!\!\!\!\!
f(x_1)f(x_2)\overline{f(x_3)}f(x_4) \, e^{2\pi i (-x_1-x_2+x_3-x_4) \xi} \,\overline{r(x_4,\xi)^2} \, \,{\rm d} x_1 \,{\rm d} x_2 \,{\rm d} x_3 \,{\rm d} x_4 \\
& - \textup{Re} \!\!\!\!\int\limits_{\{(x_1\wedge x_2)>(x_3\vee x_4)\}}\!\!\!\!\!\!
f(x_1)f(x_2)f(x_3)f(x_4) \, e^{2\pi i (-x_1-x_2-x_3-x_4) \xi} \,\overline{r(x_3,\xi)^2 r(x_4,\xi)^2} \, \,{\rm d} x_1 \,{\rm d} x_2 \,{\rm d} x_3 \,{\rm d} x_4.
\end{align*}

We now  use the first numerical inequality established in Lemma \ref{lm:inequalities} below.
Substituting
\[ u=|\widehat{f}(\xi)|^2\;\text{ and }\; v= -(\mathcal{Q} f)(\xi) + (\mathcal{E} f)(\xi) \]
into the inequality from part (a) of Lemma \ref{lm:inequalities}, and then integrating in $\xi$,
we conclude from \eqref{logexpansion} that
\begin{equation}\label{logexpansionLq}
\big\| (\log |a(\xi)|^2)^\frac{1}{2} \big\|_{\textup{L}^q_\xi(\mathbb{R})}^q
\leq \big\|\widehat{f}\big\|_{\textup{L}^q(\mathbb{R})}^q - \frac{q}{2}\mathcal{H}(f)+ \mathcal{R}(f),
\end{equation}
where the second term is given by
\[ \mathcal{H}(f) = \int_{\mathbb{R}} (\mathcal{Q} f)(\xi) |\widehat{f}(\xi)|^{q-2} \,{\rm d}\xi, \]
and the remainder $\mathcal{R}(f)$ is bounded by a linear combination (with coefficients depending only on $q$) of integrals
\begin{equation}\label{eq:integrals1}
\int_{\mathbb{R}} |(\mathcal{Q} f)(\xi)|^{\frac{q}{2}} \,{\rm d}\xi, \quad
\int_{\mathbb{R}} |(\mathcal{E} f)(\xi)|^{\frac{q}{2}} \,{\rm d}\xi, \quad
\int_{\mathbb{R}} |(\mathcal{E} f)(\xi)| |\widehat{f}(\xi)|^{q-2} \,{\rm d}\xi,
\end{equation}
and, if $q>4$,  also
\begin{equation}\label{eq:integrals2}
\int_{\mathbb{R}} |(\mathcal{Q} f)(\xi)|^2 |\widehat{f}(\xi)|^{q-4} \,{\rm d}\xi, \quad
\int_{\mathbb{R}} |(\mathcal{E} f)(\xi)|^2 |\widehat{f}(\xi)|^{q-4} \,{\rm d}\xi.
\end{equation}

Approximating the function $f$ with the Gaussian $G$ as discussed at the beginning of this section, we are thus reduced to showing
that $\mathcal{H}(G)$ is a large enough positive quantity,
that $\mathcal{H}(G)$ provides a good approximation for $\mathcal{H}(f)$,
and that the remainder term $\mathcal{R}(f)$ is appropriately small.
This is accomplished via the following sequence of lemmata, whose proofs are deferred to the next section in order not to obscure the main line of reasoning.
All of them hold under the ongoing assumption that $f$ satisfies the hypotheses of Theorem \ref{thm:main}, and that $G\in\mathfrak{G}$ approximates $f$ in the sense of \eqref{Gapproxf}.

The first lemma shows that the quantity $\mathcal{H}(G)$ is not too small.

\begin{lemma}\label{lm:Hbounds}
\begin{equation*}
\mathcal{H}(G) \gtrsim \|G\|_{\textup{L}^1(\mathbb{R})}^2 \|G\|_{\textup{L}^p(\mathbb{R})}^q.
\end{equation*}
\end{lemma}

The second lemma shows that $\mathcal{H}(G)$ provides a good approximation for $\mathcal{H}(f)$.
\begin{lemma}\label{lm:diffHbounds}
\begin{equation*}
|\mathcal{H}(f)-\mathcal{H}(G)|
\lesssim \gamma^{(q-2)\wedge 1} \|G\|_{\textup{L}^1(\mathbb{R})}^2 \|G\|_{\textup{L}^p(\mathbb{R})}^q.
\end{equation*}
\end{lemma}

The third lemma shows smallness of the remainder term $\mathcal{R}(f)$.
\begin{lemma}\label{lm:Rbounds}
\begin{equation*}
|\mathcal{R}(f)| \lesssim \delta^{(q-2)\wedge 2} \|f\|_{\textup{L}^1(\mathbb{R})}^2 \|f\|_{\textup{L}^p(\mathbb{R})}^q.
\end{equation*}
\end{lemma}

We are now in a position to finish the proof of the theorem.
If the parameter $\delta$ is chosen to be small enough, then so is $\gamma$, and the three lemmata combine with bounds \eqref{approxLp}, \eqref{approxL1} to yield
\[ \frac{q}{2} |\mathcal{H}(f)-\mathcal{H}(G)|+|\mathcal{R}(f)|\leq \frac{q}{4} \mathcal{H}(G). \]
If $\varepsilon>0$ is small enough, then this inequality together with \eqref{logexpansionLq}, Lemma~\ref{lm:Hbounds}, estimates \eqref{approxLp} and \eqref{approxL1}, and the sharp Hausdorff--Young inequality \eqref{sharpHY} imply
\begin{align*}
\big\| (\log |a(\xi)|^2)^\frac{1}{2} \big\|_{\textup{L}^q_\xi(\mathbb{R})}^q
& \leq \big\|\widehat{f}\big\|_{\textup{L}^q(\mathbb{R})}^q - \frac{q}{4}\mathcal{H}(G) \\
& \leq \big\|\widehat{f}\big\|_{\textup{L}^q(\mathbb{R})}^q - \varepsilon q {\bf B}_p^{q-1} \|f\|_{\textup{L}^1(\mathbb{R})}^2 \|f\|_{\textup{L}^p(\mathbb{R})}^q\\
& \leq \big( {\bf B}_p^q - \varepsilon q {\bf B}_p^{q-1}\|f\|_{\textup{L}^1(\mathbb{R})}^2 \big) \|f\|_{\textup{L}^p(\mathbb{R})}^q.
\end{align*}
One last application of Bernoulli's inequality finally yields
\[ \big( {\bf B}_p^q - \varepsilon q {\bf B}_p^{q-1}\|f\|_{\textup{L}^1(\mathbb{R})}^2 \big)^{\frac1q}
\leq {\bf B}_p - \varepsilon\|f\|_{\textup{L}^1(\mathbb{R})}^2, \]
as desired.
This completes the proof of Theorem \ref{thm:main} modulo the verification of the lemmata, which is the content of the next section.

\section{Proofs of lemmata}\label{sec:proofs}

We start with some elementary numerical inequalities.
\begin{lemma}\label{lm:inequalities}
Given an exponent $2<q<\infty$, the following inequalities hold.
\begin{itemize}
\item[(a)]
For $u\geq 0$ and $v\in\mathbb{R}$ one has
\[ |u+v|^{\frac{q}{2}} \leq
\left\{\begin{array}{ll}
u^{\frac{q}{2}} + \frac{q}{2} v u^{\frac{q}{2}-1} + D_q |v|^{\frac{q}{2}} & \text{if } 2<q\leq 4, \\
u^{\frac{q}{2}} + \frac{q}{2} v u^{\frac{q}{2}-1} + D_q |v|^{\frac{q}{2}} + D_q |v|^{2} u^{\frac{q}{2}-2} & \text{if } q>4,
\end{array}\right.
\]
with some finite constant $D_q>0$ depending only on $q$.
\item[(b)]
For $u\geq 0$ and $v\in\mathbb{R}$ one has
\[ \big| |u+v|^{q-2} - u^{q-2} \big| \leq
E_q \left\{\begin{array}{ll}
|v|^{q-2} & \text{if } 2<q\leq 3, \\
|v|^{q-2} + |v| u^{q-3} & \text{if } q>3,
\end{array}\right. \]
with some finite constant $E_q>0$ depending only on $q$.
\end{itemize}
\end{lemma}

\begin{proof}
(a) No generality is lost in assuming that $u\neq 0$.
We can then divide both sides of the inequality by $u^{\frac{q}{2}}$, and substituting $t=\frac vu$ we are left with checking that
\[ |1+t|^{\frac{q}{2}} - 1 - \frac{q}{2}t \lesssim_q
\left\{\begin{array}{ll}
|t|^{\frac{q}{2}} & \text{if } 2<q\leq 4, \\
|t|^{\frac{q}{2}} + t^{2} & \text{if } q>4.
\end{array}\right. \]
If $Q(t)$ denotes the quotient of the two sides of the inequality,
\[ Q(t)=
\left\{\begin{array}{ll}
\frac{|1+t|^{\frac{q}{2}} - 1 - \frac{q}{2}t}{|t|^{\frac{q}{2}}} & \text{if } 2<q\leq4, \\
\frac{|1+t|^{\frac{q}{2}} - 1 - \frac{q}{2}t}{|t|^{\frac{q}{2}}+t^{2}} & \text{if } q>4,
\end{array}\right. \]
then we need to show boundedness from above of the function $Q$ on $\mathbb{R}\setminus\{0\}$.
This is a simple consequence of the continuity of $Q$ and finiteness of the limits:
\[ \lim_{t\to0} Q(t)=
\left\{\begin{array}{ll}
0& \text{if } 2<q<4, \\
\frac{q(q-2)}{8}& \text{if } q\geq4,
\end{array}\right.
\quad \lim_{t\to\pm\infty} Q(t) = 1.\]

(b) This time the substitution $t=\frac vu$ turns the inequality into
\[ \big| |1+t|^{q-2} - 1 \big| \lesssim_q
\left\{\begin{array}{ll}
|t|^{q-2} & \text{if } 2<q\leq 3, \\
|t|^{q-2} + |t| & \text{if } q>3.
\end{array}\right. \]
If we again denote the quotient of the two sides by $Q(t)$, then the inequality follows as before from the continuity of $Q$ and finiteness of the limits:
\[ \lim_{t\to0} Q(t)=
\left\{\begin{array}{ll}
0& \text{if } 2<q<3, \\
q-2& \text{if } q\geq3,
\end{array}\right.
\quad \lim_{t\to\pm\infty} Q(t) = 1. \qedhere \]
\end{proof}

\begin{proof}[Proof of Lemma \ref{lm:Hbounds}]
Start by noting that the quotient
\begin{equation}\label{coercive}
\frac{\mathcal{H}(f)}{\|f\|_{\textup{L}^1(\mathbb{R})}^2\|f\|_{\textup{L}^p(\mathbb{R})}^q}
\end{equation}
is invariant under arbitrary scalings, modulations, translations, and $\textup{L}^1$-normalized dilations.
Indeed,
if $f(x)=c f_1(x)$ for some $c\in\mathbb{C}$, then $\mathcal{H}(f)= |c|^{q+2}\mathcal{H}(f_1)$.
Moreover,
if $f(x)=e^{2\pi i x\xi_0} f_1(x)$ for some $\xi_0\in\mathbb{R}$, then $\mathcal{H}(f)=\mathcal{H}(f_1)$.
Similarly,
if $f(x)=f_1(x-x_0)$ for some $x_0\in\mathbb{R}$, then $\mathcal{H}(f)=\mathcal{H}(f_1)$.
Finally, dilation invariance is easily seen from the Fourier representation
\[ (\mathcal{Q} f)(\xi) = \int_{\mathbb{R}} \varphi (t) e^{-2\pi i t \xi} \,{\rm d} t, \]
where the function $\varphi$ is given by
\begin{equation}\label{defvarphi1}
\varphi(t) = \beta \int\limits_{\substack{\{x_1+x_2-x_3-x_4=t,\\ (x_1\wedge x_2)>(x_3\vee x_4)\,\textup{or}\, (x_1\vee x_2)<(x_3\wedge x_4)\}}}\!\!\!\!
f(x_1)f(x_2)\overline{f(x_3)f(x_4)} \, \,{\rm d}\sigma_3(x_1,x_2,x_3,x_4).
\end{equation}
Here we are integrating over a region in the affine hyperplane $\{x_1+x_2-x_3-x_4=t\}\subset \mathbb{R}^4$ with respect to the $3$-dimensional Hausdorff measure $\sigma_3$.
The constant $\beta>0$ is unimportant and it is coming from the non-orthonormal choice of coordinates.
It follows that, if $f(x)=\lambda^{-1}f_1(\lambda^{-1}x)$, then
\[ \varphi(t) = \frac{1}{\lambda} \varphi_1\Big(\frac{t}{\lambda}\Big),\quad (\mathcal{Q} f)(\xi) = (\mathcal{Q} f_1)(\lambda\xi),\quad \mathcal{H}(f) = \lambda^{-1}\mathcal{H}(f_1). \]
On the other hand,
\[ \|f\|_{\textup{L}^1(\mathbb{R})}^{2} \|f\|_{\textup{L}^p(\mathbb{R})}^{q}
= \lambda^{-1} \|f_1\|_{\textup{L}^1(\mathbb{R})}^{2} \|f_1\|_{\textup{L}^p(\mathbb{R})}^{q}. \]
This shows that the expression \eqref{coercive} is invariant under $\textup{L}^1$-normalized dilations, as claimed.
Now, going back to \eqref{Gaussian} and writing
\[ G(x)=c e^{-\alpha x^2+ v x}=ce^{\frac{\textup{Re}(v)^2}{4\alpha}}e^{i \textup{Im}(v) x}e^{-\alpha(x-\frac{\textup{Re}(v)}{2\alpha})^2}, \]
one sees that any Gaussian can be brought to standard form by an application of an appropriate scaling, modulation, translation, and dilation.
Given the symmetries of \eqref{coercive} just discussed, in verifying the claim of the lemma we can assume that the Gaussian approximation $G$ coincides with the standard Gaussian, $G(x)=\exp(-\pi x^2)$.
In this case, we are reduced to checking that $\mathcal{H}(G)>0$.
Again writing $\mathcal{Q} G$ as the Fourier transform of some function $\varphi$, then from formula \eqref{defvarphi1} above with $f=G$ it immediately follows that $\varphi$ is nonnegative and not identically zero.
Also, $|\widehat{G}(\xi)|^{q-2}$ is the Fourier transform of another Gaussian function $\psi$, given by
\[\psi(x)=\frac{1}{\sqrt{q-2}}e^{-\frac{\pi x^2}{q-2}},\]
which is clearly strictly positive.
It remains to invoke unitarity of the linear Fourier transform and observe that
\[ \mathcal{H}(G) = \big\langle \mathcal{Q}G, |\widehat{G}|^{q-2} \big\rangle_{\textup{L}^2(\mathbb{R})}
= \langle\varphi,\psi\rangle_{\textup{L}^2(\mathbb{R})} >0. \]
This concludes the proof of the lemma.
\end{proof}

The proofs of the two remaining lemmata rely on the observation that $\mathcal{Q} f$ can be pointwise controlled by the maximally truncated Fourier transform $\mathcal{F}_\ast f$, defined in \eqref{MaxTruncFT}.
Indeed,
\begin{align}
|(\mathcal{Q} f)(\xi)|
& \leq \int_{\mathbb{R}^2} \sup_{x\in\mathbb{R}}\Big|\int_{x}^{+\infty} f(y) e^{-2\pi i y\xi} \,{\rm d} y\Big|^2 |f(x_3)| |f(x_4)| \,{\rm d} x_3 \,{\rm d} x_4\notag\\
& = \|f\|_{\textup{L}^1(\mathbb{R})}^2 (\mathcal{F}_\ast f)^2(\xi). \label{PhiFstar}
\end{align}
In a similar way, $\mathcal{E}f$ can be  controlled pointwise by $\mathcal{F}_\ast f$.
To see this, start by noting that the reflection coefficient satisfies $|r(x,\xi)|\leq 1$. It then follows from the integral equation \eqref{intRicatti} that also $|r(x,\xi)|\leq2\|f\|_{\textup{L}^1(\mathbb{R})}$. This time we get
\begin{align}
|(\mathcal{E}f)(\xi)|
& \leq 3 \int_{\mathbb{R}^2} \sup_{x\in\mathbb{R}}\Big|\int_{x}^{+\infty} f(y) e^{-2\pi i y\xi} \,{\rm d} y\Big|^2 |f(x_3)| |f(x_4)| |r(x_4,\xi)|^2 \,{\rm d} x_3 \,{\rm d} x_4\notag\\
& \leq 12 \|f\|_{\textup{L}^1(\mathbb{R})}^4 (\mathcal{F}_\ast f)^2(\xi). \label{EFstar}
\end{align}

\begin{proof}[Proof of Lemma \ref{lm:diffHbounds}]
We start by rewriting $\mathcal{H}(f)-\mathcal{H}(G)$ as
\begin{equation}\label{mainerror}
\int_\mathbb{R} (\mathcal{Q} f)(\xi) \,\big(|\widehat{f}(\xi)|^{q-2}-|\widehat{G}(\xi)|^{q-2}\big)\,{\rm d}\xi
+ \int_\mathbb{R} \big((\mathcal{Q} f)(\xi)-(\mathcal{Q} G)(\xi)\big)\, |\widehat{G}(\xi)|^{q-2}\,{\rm d}\xi.
\end{equation}
Set $h=f-G$. Part (b) of Lemma \ref{lm:inequalities} allows us to bound the first integral in \eqref{mainerror} by a multiple of
\[ \int_\mathbb{R} |(\mathcal{Q} f)(\xi)||\widehat{h}(\xi)|^{q-2}\,{\rm d}\xi \;\textrm{ and, if } q>3, \textrm{ also }
\int_\mathbb{R} |(\mathcal{Q} f)(\xi)||\widehat{h}(\xi)||\widehat{G}(\xi)|^{q-3}\,{\rm d}\xi. \]
To bound these integrals, start by observing that estimates \eqref{PhiFstar} and \eqref{MPZ} imply
\[ \int_\mathbb{R} |(\mathcal{Q} f)(\xi)|^{\frac q2} \,{\rm d}\xi
\leq \|f\|_{\textup{L}^1(\mathbb{R})}^q \int_\mathbb{R}(\mathcal{F}_\ast f)^q(\xi)\,{\rm d}\xi
\lesssim \|f\|_{\textup{L}^1(\mathbb{R})}^q\|f\|_{\textup{L}^p(\mathbb{R})}^q. \]
H\"older's inequality and the Hausdorff--Young inequality then imply
\begin{align*}
\int_\mathbb{R} |(\mathcal{Q} f)(\xi)||\widehat{h}(\xi)|^{q-2}\,{\rm d}\xi
&\leq
\Big(\int_\mathbb{R} |(\mathcal{Q} f)(\xi)|^{\frac q2}\,{\rm d}\xi\Big)^{\frac 2q}\Big(\int_\mathbb{R} |\widehat{h}(\xi)|^{q}\,{\rm d}\xi\Big)^{1-\frac{2}q}\\
&\lesssim
\|f\|_{\textup{L}^1(\mathbb{R})}^2\|f\|_{\textup{L}^p(\mathbb{R})}^2\|h\|_{\textup{L}^p(\mathbb{R})}^{q-2}\\
&\lesssim
\gamma^{q-2} \|G\|_{\textup{L}^1(\mathbb{R})}^2\|G\|_{\textup{L}^p(\mathbb{R})}^q,
\end{align*}
where the last inequality is a consequence of \eqref{Gapproxf}, \eqref{approxLp}, and \eqref{approxL1}.
In particular, this term is acceptable.
In a similar way,
\[ \int_\mathbb{R} |(\mathcal{Q} f)(\xi)||\widehat{h}(\xi)||\widehat{G}(\xi)|^{q-3}\,{\rm d}\xi
\lesssim
\|f\|_{\textup{L}^1(\mathbb{R})}^2 \|f\|_{\textup{L}^p(\mathbb{R})}^2 \|h\|_{\textup{L}^p(\mathbb{R})} \|G\|_{\textup{L}^p(\mathbb{R})}^{q-3}
\lesssim
\gamma \|G\|_{\textup{L}^1(\mathbb{R})}^2\|G\|_{\textup{L}^p(\mathbb{R})}^q \]
is likewise acceptable for $q>3$. Now we focus on the second integral from \eqref{mainerror}.
It is useful to introduce the quadrilinear operator $\Phi$ by
\begin{align*}
& \Phi(f_1,f_2,f_3,f_4)(\xi) \\
& = \textup{Re} \int\limits_{\{(x_1\wedge x_2)>(x_3\vee x_4)\}}\!\!
f_1(x_1)f_2(x_2) \overline{f_3(x_3)f_4(x_4)} \, e^{2\pi i (-x_1-x_2+x_3+x_4) \xi} \, \,{\rm d} x_1 \,{\rm d} x_2 \,{\rm d} x_3 \,{\rm d} x_4,
\end{align*}
so that $\mathcal{Q}(f)=\Phi(f,f,f,f)$.
It can be estimated either by
\begin{align}
& |\Phi(f_1,f_2,f_3,f_4)(\xi)| \nonumber \\
& \leq \int_{\mathbb{R}^2} |f_1(x_1)| |f_2(x_2)| \,\sup_{x\in\mathbb{R}}\Big|\int_{-\infty}^{x} f_3(y) e^{-2\pi i y\xi} \,{\rm d} y\Big| \,\sup_{x\in\mathbb{R}}\Big|\int_{-\infty}^{x} f_4(y) e^{-2\pi i y\xi} \,{\rm d} y\Big| \,{\rm d} x_1 \,{\rm d} x_2 \nonumber \\
& = \|f_1\|_{\textup{L}^1(\mathbb{R})} \|f_2\|_{\textup{L}^1(\mathbb{R})} (\mathcal{F}_\ast f_3)(\xi) (\mathcal{F}_\ast f_4)(\xi), \label{PhiFstar1}
\end{align}
or (similarly as $\mathcal{Q}$) by
\begin{equation}\label{PhiFstar2}
|\Phi(f_1,f_2,f_3,f_4)(\xi)| \leq \|f_3\|_{\textup{L}^1(\mathbb{R})} \|f_4\|_{\textup{L}^1(\mathbb{R})} (\mathcal{F}_\ast f_1)(\xi) (\mathcal{F}_\ast f_2)(\xi).
\end{equation}
Multilinearity implies
\[ \mathcal{Q}f-\mathcal{Q}G
= \Phi(G,G,G,h) + \Phi(G,G,h,f) + \Phi(G,h,f,f) + \Phi(h,f,f,f), \]
so the $\textup{L}^{\frac{q}{2}}$-norm of this function is controlled using \eqref{PhiFstar1}, \eqref{PhiFstar2}, and H\"older's inequality by
\begin{align*}
\|\mathcal{Q}f-\mathcal{Q}G\|_{\textup{L}^\frac{q}{2}(\mathbb{R})}
& \leq \|G\|_{\textup{L}^1(\mathbb{R})}^2 \|\mathcal{F}_\ast G\|_{\textup{L}^q(\mathbb{R})} \|\mathcal{F}_\ast h\|_{\textup{L}^q(\mathbb{R})}
+ \|G\|_{\textup{L}^1(\mathbb{R})}^2 \|\mathcal{F}_\ast h\|_{\textup{L}^q(\mathbb{R})} \|\mathcal{F}_\ast f\|_{\textup{L}^q(\mathbb{R})} \\
& \quad + \|f\|_{\textup{L}^1(\mathbb{R})}^2 \|\mathcal{F}_\ast G\|_{\textup{L}^q(\mathbb{R})} \|\mathcal{F}_\ast h\|_{\textup{L}^q(\mathbb{R})}
+ \|f\|_{\textup{L}^1(\mathbb{R})}^2 \|\mathcal{F}_\ast h\|_{\textup{L}^q(\mathbb{R})} \|\mathcal{F}_\ast f\|_{\textup{L}^q(\mathbb{R})}.
\end{align*}
Invoking the Menshov--Paley--Zygmund inequality, \eqref{Gapproxf}, \eqref{approxLp}, and \eqref{approxL1} then gives
\[ \|\mathcal{Q}f-\mathcal{Q}G\|_{\textup{L}^\frac{q}{2}(\mathbb{R})} \lesssim \gamma \|G\|_{\textup{L}^1(\mathbb{R})}^2 \|G\|_{\textup{L}^p(\mathbb{R})}^2. \]
Finally, yet another application of H\"older's inequality combined with the linear Hausdorff--Young inequality for $G$ bounds the second integral in \eqref{mainerror} by a constant times
\[ \gamma \|G\|_{\textup{L}^1(\mathbb{R})}^2 \|G\|_{\textup{L}^p(\mathbb{R})}^q. \]
As noted before, this concludes the proof of the lemma.
\end{proof}

\begin{proof}[Proof of Lemma \ref{lm:Rbounds}]
The proof parallels that of the previous lemma, and so we shall be brief.
Using H\"older's inequality, the pointwise estimates \eqref{PhiFstar} and \eqref{EFstar},  and that fact that the maximally truncated Fourier transform $\mathcal{F}_\ast$ satisfies the estimate \eqref{MPZ}, we conclude that
the five integrals in \eqref{eq:integrals1} and \eqref{eq:integrals2} are  respectively controlled by
\begin{align*}
& \|\mathcal{Q} f\|_{\textup{L}^{\frac q2}(\mathbb{R})}^{\frac q2} \leq \|f\|_{\textup{L}^1(\mathbb{R})}^q \|\mathcal{F}_\ast f\|_{\textup{L}^q(\mathbb{R})}^q
\lesssim \|f\|_{\textup{L}^1(\mathbb{R})}^q \|f\|_{\textup{L}^p(\mathbb{R})}^q, \\
& \|\mathcal{E}f\|_{\textup{L}^{\frac q2}(\mathbb{R})}^{\frac q2} \lesssim \|f\|_{\textup{L}^1(\mathbb{R})}^{2q} \|\mathcal{F}_\ast f\|_{\textup{L}^q(\mathbb{R})}^q
\lesssim \|f\|_{\textup{L}^1(\mathbb{R})}^{2q} \|f\|_{\textup{L}^p(\mathbb{R})}^q, \\
& \|\mathcal{E}f\|_{\textup{L}^{\frac q2}(\mathbb{R})} \|\widehat{f}\|_{\textup{L}^q(\mathbb{R})}^{q-2}
\lesssim \big( \|f\|_{\textup{L}^1(\mathbb{R})}^{2q} \|f\|_{\textup{L}^p(\mathbb{R})}^q \big)^{\frac 2q} \|f\|_{\textup{L}^p(\mathbb{R})}^{q-2}
= \|f\|_{\textup{L}^1(\mathbb{R})}^4 \|f\|_{\textup{L}^p(\mathbb{R})}^q, \\
& \|\mathcal{Q} f\|_{\textup{L}^{\frac q2}(\mathbb{R})}^2 \|\widehat{f}\|_{\textup{L}^q(\mathbb{R})}^{q-4}
\lesssim \big( \|f\|_{\textup{L}^1(\mathbb{R})}^q \|f\|_{\textup{L}^p(\mathbb{R})}^q \big)^{\frac4q} \|f\|_{\textup{L}^p(\mathbb{R})}^{q-4}
= \|f\|_{\textup{L}^1(\mathbb{R})}^4 \|f\|_{\textup{L}^p(\mathbb{R})}^q, \\
& \|\mathcal{E}f\|_{\textup{L}^{\frac q2}(\mathbb{R})}^2 \|\widehat{f}\|_{\textup{L}^q(\mathbb{R})}^{q-4}
\lesssim \big( \|f\|_{\textup{L}^1(\mathbb{R})}^{2q} \|f\|_{\textup{L}^p(\mathbb{R})}^q \big)^{\frac4q} \|f\|_{\textup{L}^p(\mathbb{R})}^{q-4}
= \|f\|_{\textup{L}^1(\mathbb{R})}^8 \|f\|_{\textup{L}^p(\mathbb{R})}^q.
\end{align*}
Since $\|f\|_{\textup{L}^1(\mathbb{R})}\leq\delta$, the result follows.
\end{proof}

\section*{Acknowledgements}
V.K. was supported in part by the Croatian Science Foundation under the project 3526.
V.K. and J.R. were partially supported by the bilateral DAAD-MZO grant \emph{Multilinear singular integrals and applications}.
D.O.S. was partially supported by the Hausdorff Center for Mathematics and DFG grant CRC1060.
This work was started during a pleasant visit of the second author to the University of Zagreb, whose hospitality is greatly appreciated.
The authors are indebted to the anonymous referee for careful reading and valuable suggestions.

\bibliographystyle{alpha}

\end{document}